\documentclass[a4paper,11pt, reqno]{amsart}
\usepackage[margin=3cm]{geometry}
\usepackage{amsmath,amsfonts,amssymb}

\usepackage{verbatim}
\usepackage{enumerate}
\usepackage{amsthm}
\usepackage{url}
\usepackage{multicol}
\usepackage[colorlinks]{hyperref}
\usepackage[english]{babel}
\usepackage{tikz}
\usepackage{mathtools}
\usepackage[utf8]{inputenc}
\usepackage{xcolor}

\newcommand{\Z}{\mathbb Z}

\newcommand{\Tr}{\mathrm{Tr}}
\newcommand{\fqn}{\mathbb{F}_{q^n}}
\newcommand{\F}{\mathbb{F}}

\newcommand{\fq}{\mathbb{F}_q}
\newcommand{\p}{\mathcal{P}}

\newtheorem{theorem}{Theorem}[section]

\newtheorem{proposition}[theorem]{Proposition}
\newtheorem{definition}[theorem]{Definition}

\newtheorem{lemma}[theorem]{Lemma}
\newtheorem{corollary}[theorem]{Corollary}
\newtheorem{example}[theorem]{Example}
\newtheorem{remark}[theorem]{Remark}

 \author[F. E. Brochero Mart\'inez] {F.E. Brochero Mart\'{i}nez}
\address{Departamento de Matem\'{a}tica,
UNIVERSIDADE FEDERAL DE MINAS GERAIS,
UFMG,
Belo Horizonte MG (Brazil)}
\email{fbrocher@mat.ufmg.br}

\author{Daniela  Oliveira}
\address{Instituto de Ciências Matemáticas e de Computação, UNIVERSIDADE DE SÃO PAULO (USP)
São Carlos (Brazil)
The second author was supported by FAPESP 2022/14004-7, Brazil.}

\email{danielaalvesoliveira@gmail.com}


\title[On the number of rational points  of Artin-Schreier's curves $\dots$]
{On the number of rational points  of Artin-Schreier's curves and hypersurfaces}

\begin{document} 

\maketitle 
\begin{abstract}
Let $\mathbb{F}_{q^n}$ represent the finite field with $q^n$ elements. In this  paper, our focus is on determining the number of $\mathbb{F}_{q^n}$-rational points for two specific objects: an affine Artin-Schreier curve given by the equation $y^q-y = x(x^{q^i}-x)-\lambda$, and an Artin-Schreier hypersurface given by the equation $y^q-y=\sum_{j=1}^r a_jx_j(x_j^{q^{i_j}}-x_j)-\lambda$.

Additionally, we establish that the Weil bound is only achieved in these cases when the trace of the element $\lambda\in\mathbb{F}_{q^n}$ over the subfield $\mathbb{F}_q$ is equal to zero.
\end{abstract}

\section{Introduction}
Let $\mathbb{F}_q$ denote a finite field with $q=p^s$ elements, where $p$ is an odd prime. One important class of curves over finite fields is known as Artin-Schreier's curves. These curves are defined by the equation $y^q-y=f(x)$, where $f(x)\in \mathbb{F}_q[x]$. Extensive research has been conducted on these curves in various contexts, including references such as \cite{COS, Coulter, FP, H, Van}.

This class of curves can be generalized to multiple variables, resulting in hypersurfaces of the form $y^q-y=f(X)$, where $f(X) \in \mathbb{F}_q[X]$ and $X=(x_1,\dots,x_r)$. The study of the number of affine rational points on algebraic hypersurfaces over finite fields has significant applications in coding theory, cryptography, communications, and related fields, as evidenced by works such as \cite{Aubry, GO, Stepanov, Van}.

The first aim of this paper is to calculate the number of $\mathbb{F}_{q^n}$-rational points on the Artin-Schreier curve $\mathcal{C}_i$ defined by the equation:
\[\mathcal C_i: y^q-y = x(x^{q^i}-x) -\lambda,\]
where $i \in \mathbb{N}$ and $\lambda \in \mathbb{F}_{q^n}$. We denote the number of $\mathbb{F}_{q^n}$-rational points of $\mathcal{C}_i$ as $N_n(\mathcal{C}_i)$.

For each $i \in \mathbb{N}$, we introduce the map $Q_i: \mathbb{F}_{q^n} \rightarrow \mathbb{F}_q$ defined as follows
\begin{align*}
Q_i : \fqn &\to \fq\\
 \alpha & \mapsto \Tr(\alpha(\alpha^{q^i}-\alpha)-\lambda),
\end{align*}
where $\mathrm{Tr}:\fqn \to \fq$ denotes the trace function. We define $N_n(Q_i)$ as the number of zeroes of $Q_i$ in $\mathbb{F}_{q^n}$. According to Hilbert's Theorem 90, we can establish the relationship
\[N_n(\mathcal C_{i}) = q \cdot N_n(Q_i).\]
Thus, determining $N_n(\mathcal{C}_i)$ is equivalent to calculate $N_n(Q_i)$. Additional details regarding this fact can be found in  \cite{A1, A2, OS}.

The study of the number of $\mathbb{F}_{q^n}$-rational points on Artin-Schreier's curves has attracted the attention of many authors. 
For instance,  Wolfmann determined in  \cite{Wolf1} the number of rational points on the curve defined over $\mathbb{F}_{q^k}$, given by the equation $y^q - y = ax^s + b$, where $a \in \mathbb{F}_{q^k}^*$, $b \in \mathbb{F}_{q^k}$, $k$ is an even special integer, and $s$ has specific properties. 
Coulter, in \cite{Coulter}, determined the number of $\mathbb{F}_q$-rational points on the curve $y^{p^n} - y = ax^{p^{\alpha}+1} + L(x)$, where $a \in \mathbb{F}_q^*$, $L(x)$ is a $\mathbb{F}_{p^t}$-linearized polynomial, and $t = \gcd(n,s)$ divides $d = \gcd(\alpha,s)$. 
In  \cite{COS}, the authors determined the number of $\mathbb{F}_{q^m}$-rational points on the curve $y^{q^n} - y = \gamma x^{q^h+1} - \alpha$, in specific cases when $h, n, m$ are positive integers with $n$ dividing $m$, and $\gamma, a \in \mathbb{F}_{q^m}$ ($\gamma \neq 0$).
 In \cite{Lu}, they provided specific examples of maximal and minimal curves over finite fields in odd characteristic. They studied  the Artin-Schreier curves of the form $y^q - y = xS(x)$, where $S(x)$ represents a $\mathbb{F}_q$-linearized polynomial, and established a connection with quadratic forms. In our prior research \cite{BrocheroOliveira}, we obtained a formula for $N_n(C_i)$ when $\gcd(n,p)=1$.

In this paper, we employ an alternative method that involves determining the solutions of the quadratic form $Q_i$ using appropriate permutation matrices. With this approach, we compute the number $N_n(\mathcal{C}_i)$ without imposing any conditions on $\gcd(n,p)$. Additionally, we establish conditions for the Artin-Schreier curve $C_i$ in order to be maximal or minimal with respect to the Hasse-Weil bound.

The second objective of this paper is to determine the number of $\mathbb{F}_{q^n}$-rational points on the affine Artin-Schreier hypersurface $\mathcal{H}_r$, which is given by the equation

\begin{equation} \label{hr}
\mathcal{H}_r: y^q - y = \sum_{j=1}^r a_jx_j(x_j^{q^{i_j}} - x_j) - \lambda,
\end{equation}
where $a_j\in\mathbb{F}_q^*$ and $0< i_j< n$ for $1\le j\le r$. We denote by $N_n(\mathcal{H}_r)$ the number of $\mathbb{F}_{q^n}$-rational points on the hypersurface $\mathcal{H}_r$. The well-known Weil bound assures us that

\begin{equation}\label{cota}
|N_n(\mathcal{H}_r)-q^{rn}| \le (q-1)\prod_{j=1}^r q^{i_j}q^{\frac{nr}{2}}= (q-1)q^{\frac{nr+2I}{2}},
\end{equation}
where $I = \sum_{j=1}^r i_j$.

In this paper, we present necessary and sufficient conditions in order to the hypersurface $\mathcal{H}_r$ to be $\mathbb{F}_{q^n}$-maximal or $\mathbb{F}_{q^n}$-minimal, i.e.  $N_n(\mathcal{H}_r)$ achieves the upper or lower bound specified in \eqref{cota}. Currently, in the literature is always interesting to given a description of the number of rational points for Artin-Schreier's hypersurfaces, as well as the conditions under which these hypersurfaces satisfy the bound given in  \eqref{cota} or other known bounds.

Using the results obtained for the curve $\mathcal{C}_i$, we can explicitly determine the number of rational points $N_n(\mathcal{H}_r)$ and to identify the conditions to achieve the maximality or minimality of this hypersurface.

The organization of this paper is as follows: Section 2 provides background material and preliminary results. In Section 3, we compute the number of $\mathbb{F}_{q^n}$-rational points on the Artin-Schreier curves $C_i$ and establish necessary and sufficient conditions for these curves to be maximal or minimal. In Section 4, we determine the number of $\mathbb{F}_{q^n}$-rational points on the hypersurfaces $\mathcal{H}_r$ in  \eqref{hr} and derive explicit conditions for these hypersurfaces to be maximal or minimal.

\section{Preliminary results}
In this paper, we use the symbols $\psi$ and $\tilde{\psi}$ to represent the canonical additive characters of $\mathbb{F}_{q^n}$ and $\mathbb{F}_q$, respectively. The quadratic character of $\mathbb{F}_q$ is denoted as $\chi$. The trace function $\Tr$ maps elements from $\mathbb{F}_{q^n}$ to $\mathbb{F}_q$ and is defined as follows:
\begin{align*}
\Tr:  \fqn \to & \, \fq\\
x \mapsto &  \, x+x^q+\cdots + x^{q^{n-1}}.
\end{align*}
We use $\tau$ to represent the function
$\tau = \begin{cases}
	1& \text{ if } p \equiv 1 \pmod 4;\\
	i & \text{ if } p \equiv 3 \pmod 4.
	\end{cases}$

To determine the number of rational points on the curve $\mathcal{C}_i$, we associate the quadratic form $\Tr(x(x^{q^i}-x))$ with this curve. By fixing a basis for $\mathbb{F}_{q^n}$ over $\mathbb{F}_q$, we can provide its associated matrix and the dimension of its radical. To achieve this, we need to recall the following definitions.

 \begin{definition}
Consider a quadratic form $Q:\mathbb{F}_{q^n} \to \mathbb{F}_q$. The associated symmetric bilinear form $B:\mathbb{F}_{q^n}\times \mathbb{F}_{q^n} \to \mathbb{F}_q$ of $Q$ is defined as follows
$$B(\alpha,\beta) = \frac{1}2\left(Q(\alpha+\beta) -Q(\alpha)-Q(\beta)\right).$$ 
The radical of the quadratic form $Q: \mathbb{F}_{q^n} \to \mathbb{F}_q$ is the $\mathbb{F}_q$-subspace defined as
\[\text{rad}(Q) = \{\alpha \in \fqn : B(\alpha, \beta) = 0 \text{ for all } \beta\in \fqn \}.\]
Furthermore, $Q$ is considered non-degenerate if $\text{rad}(Q) = \{0\}$.
\end{definition}

Let $\mathcal{B}=\{v_1,\dots, v_n\}$ be a basis for $\mathbb{F}_{q^n}$ over $\mathbb{F}_q$. The $n\times n$ matrix $A=(a_{ij})$, defined as $a_{ij}= B(v_i, v_j)$, represents the associated matrix of the quadratic form $Q$ with respect to the basis $\mathcal{B}$. Notably, the dimension of $\text{rad}(Q)$ is given by $n - \text{rank}(A)$.

Consider the quadratic forms $Q_1:\mathbb{F}_{q^m} \to \mathbb{F}_q$ and $Q_2:\mathbb{F}_{q^n} \to \mathbb{F}_q$, where $m \geq n$. Let $U$ and $V$ be the associated matrices of $Q_1$ and $Q_2$, respectively. We say that $Q_1$ is equivalent to $Q_2$ if there exists $M \in \text{GL}_m(\mathbb{F}_q)$ such that
\[M^T UM =\left(
\begin{array}{c|c}
V & 0 \\ \hline
 0 & 0
\end{array}\right)\in M_m(\fq),\]

where $\text{GL}_m(\mathbb{F}_q)$ represents the group of invertible $m\times m$ matrices over $\mathbb{F}_q$, and $M_m(\mathbb{F}_q)$ denotes the set of $m\times m$ matrices over $\mathbb{F}_q$. Furthermore, $Q_2$ is said to be a {\em reduced form} of $Q_1$ if $\text{rad}(Q_2)= \{0\}$.
The following theorem is a well-known result about the number of solutions of quadratic forms over finite fields.

\begin{theorem}[\cite{LiNi}, Theorems $6.26$ and $6.27$] \label{sol}
Let $Q$ be a quadratic form over $\fqn$, where $q$ a power of an odd prime. Let $B_Q$ be the bilinear symmetric form associated to  $Q$,  $v = \dim( \text{rad}(B_Q))$ and $\tilde Q$ a reduced nondegenerate quadratic form equivalent to $Q$.  Set  $S_{\alpha}  = |\{ x \in \fqn| Q(x) = \alpha\}|$ and  let $\Delta$ be the determinant of the quadratic form $\tilde Q$.   Then 
\begin{enumerate}[(i)]
\item If  $n+v$ is even
\begin{equation} \label{casopar}
S_{\alpha} =
 \left\{
\begin{array}{ll}  
q^{n-1} +Dq^{(n+v-2)/2}(q-1) \quad &\text{ if } \alpha = 0,\\
 q^{n-1} - Dq^{(n+v-2)/2} \quad &\text{ if } \alpha \neq 0,
\end{array}\right.
\end{equation}
where $D = \chi((-1)^{(n-v)/2}\Delta)$. 

\item If $n+v$ is odd
\begin{align} \label{casoimpar}
S_{\alpha} = \left\{
\begin{array}{ll}  
q^{n-1}  \quad &\text{ if } \alpha = 0,\\
q^{n-1} + Dq^{(n+v-1)/2} \quad &\text{ if } \alpha \neq 0.
\end{array}\right.
\end{align}
where $D = \chi((-1)^{(n-v-1)/2}\alpha\Delta)$.
\end{enumerate}
In particular $D \in \{-1,1\}.$
\end{theorem}
The following lemma establishes a connection between quadratic forms and character sums, providing valuable insights for our results. It can be derived through a straightforward computation using Theorem \ref{sol}.

\begin{lemma}\label{soma}
Let $H$ be a $n\times n$ non null symmetric matrix over $\fq$ and $l = \text{rank}(H)$. Then there exists $M \in \text{GL}_n(\fq)$ such that $D = MHM^T$ is a diagonal matrix, i.e., $D = \text{diag}(a_1,a_2, \dots, a_l, 0,\dots,0)$ where $a_i \in \fq^*$ for all $i=1,\dots,l$. For the quadratic form 
\[F: \fq^n \to \fq , \quad F(X) = XHX^T \quad (X = (x_1, \dots, x_n) \in \fq^n), \]
it follows that
$$\sum_{X \in \fqn} \tilde\psi\big(F(X)\big)= (-1)^{l(s+1)}\tau^{ls}\chi(\delta)q^{n-l/2},$$
where $\delta = a_1 \cdots a_l$. 
\end{lemma}

We end this section with some basics definitions and results on Gauss sums.
\begin{definition}
Let $\Psi$ be an additive character of $\fq$ and $\Phi$ a multiplicative character  of $\fq^*$. The Gauss sum of $\Psi $ and $\Phi$ is defined by
\[G(\Psi,\Phi) = \sum_{x \in \fq^*} \Psi(x)\Phi(x).\]
\end{definition}

\begin{theorem}\cite[Theorem 5.15]{LiNi}\label{gaussdois}
	Let $\chi$ be the quadratic character of $\fq^*$. Then
	$$G(\tilde\psi,\chi)=-(-\tau)^{s}\sqrt{q}.$$
\end{theorem}

\section{The curve $y^q-y=x(x^{q^i}-x)-\lambda$}
In \cite{BrocheroOliveira}, we establish a formula for computing the number $N_n(\mathcal C_1)$ when $\gcd(n,p)=1$. In this section, we introduce a method that enables us to compute the number $N_n(\mathcal C_i)$ without imposing any conditions on $n$ and an odd prime $p$.
Consider the basis $\mathcal B=\{\beta_1, \dots, \beta_n\}$ of $\mathbb{F}_{q^n}$ over $\mathbb{F}_q$ and
\begin{equation}\label{B}
B =\begin{psmallmatrix}
\beta_1 & \beta_1^q &\cdots  &\beta_1^{q^{n-1}} \\
\beta_2 & \beta_2^q & \cdots & \beta_2^{q^{n-1}} \\
\vdots & \cdots & \ddots & \vdots \\
\beta_n & \beta_n^q & \cdots & \beta_n^{q^{n-1}}
\end{psmallmatrix}.
\end{equation}
We utilize this matrix as a tool to establish a connection between $N_n(\mathcal C_{i})$ and the number of solutions of the equation
\begin{equation} \label{Tra1}
\Tr (\alpha(\alpha^{q^i} -\alpha)) = \Tr(\lambda),
\end{equation}
where $\alpha \in \mathbb{F}_{q^n}$ and $\lambda$ is a fixed element. Here, $N_n(Q_i)$ represents the number of solutions of $\Tr(\alpha(\alpha^{q^i} -\alpha))=\Tr(\lambda)$ for $\alpha \in \mathbb{F}_{q^n}$. By Hilbert's Theorem $90$, we know that
$$ N_n(\mathcal C_{i}) = q N_n(Q_i).$$ 
The following proposition establishes a connection between $\Tr(x^{q^i+1}-x^2-\lambda)$ and a quadratic form.
The following proposition associates $\Tr(x^{q^i+1}-x^2-\lambda)$ with  a quadratic form. 

\begin{proposition}\label{pr2}
Let $f(x) = x^{q^i+1}-x^2-\lambda$, where $\lambda \in \fqn$. The number of solutions of $\Tr(f(x))=0$ in $\F_{q^n}$ is equal to the number of solutions in $\F_q^n$  of the quadratic form 
\[(x_1 \,\, x_2 \,\, \cdots \,\, x_n) A  \begin{pmatrix}
x_1\\
x_2 \\
\vdots \\
x_n
\end{pmatrix} = \Tr(\lambda), \] where   $A = (a_{j,l})$ is the $n\times n$ matrix defined by  the relations $a_{j,l} =   \frac{1}{2} \Tr(\beta_j^{q^i}\beta_l +\beta_l^{q^i} \beta_j -2\beta_j\beta_l).$  
\end{proposition}

\begin{proof}
Given $x\in \mathbb{F}_{q^n}$, we can express it as $x = \sum_{j=1}^{n} \beta_jx_j$, where $x_1,\dots, x_n\in \mathbb{F}_q$ are coefficients with respect to the basis $\mathcal{B}$. Consequently, the equation $\Tr(f(x)) = 0$ is equivalent to
\begin{equation} \label{tra}
\sum_{k=0}^{n-1} f(x)^{q^k} = 0.
\end{equation}
Expanding this expression, we obtain
\begin{align*}
\sum_{k=0}^{n-1} f(x)^{q^k}  & = \sum_{k=0}^{n-1}\Bigl( \sum_{j=1}^{n} \beta_j x_j\Bigr)^{q^k} \Bigl( \sum_{l=1}^{n} (\beta_l^{q^i}-\beta_l) x_l \Bigr)^{q^k}-\Tr(\lambda)\\
&= \sum_{j,l=1}^{n} \Bigl( \sum_{i=0}^{n-1} \beta_j^{q^k}(\beta_l^{q^{i+k}}-\beta_l^{q^k})\Bigr) x_jx_l -\Tr(\lambda).
\end{align*}
The relation 
 \(\displaystyle\sum_{j,l=1}^{n} \left( \sum_{k=0}^{n-1} \beta_j^{q^k}(\beta_l^{q^{i+k}}-\beta_l^{q^k})\right) x_jx_l = \Tr(\lambda)\) is obtained  from \eqref{tra}. 
 By symmetrizing this expression, we observe that

 $$  \frac{1}{2}\displaystyle   \sum_{k=0}^{n-1} \beta_j^{q^k}\beta_l^{q^k}(\beta_l^{q^{i+k}-q^k}+\beta_j^{q^{i+k}-q^k}-2)= \frac{1}2 \Tr\big(\beta_j^{q^i}\beta_l+\beta_l^{q^i}\beta_j-2\beta_j\beta_l \big)=a_{j,l}, $$
 and the result follows.
\end{proof}

The matrix $A$ in Proposition \ref{pr2} can be expressed as
\[A = \frac{1}{2} (A_1+A_2-2A_3),\]
where  $A_1 = (\Tr(\beta_j^{q^i}\beta_l))_{j,l}$,  $A_2 = (\Tr(\beta_j\beta_l^{q^i}))_{j,l}$ and $A_3 = (\Tr(\beta_j\beta_l))_{j,l}$.  
Let $\mathcal{P}$ be the $n\times n$ cyclic permutation matrix defined by
\[\mathcal P =  \begin{psmallmatrix}
0&1&0& \cdots & 0\\
0&0&1& \cdots &0 \\
\vdots& \cdots & \ddots & \cdots &\vdots\\
0&0&0& \cdots &1\\
1&0&0& \cdots & 0
\end{psmallmatrix},\]
and note that $\mathcal{P}^{-1} = \mathcal{P}^T$. It follows that $A_1 = B(\mathcal{P}^i)^T B^T$, $A_2 = B(\mathcal{P}^i) B^T$, and $A_3 = BB^T$. Hence, we have $A = \frac{1}{2} BM_{n,i}B^T$, where
 
\begin{equation}
M_{n,i} = \left(\mathcal{P}^i\right)^{^{\ }_T}-2 Id +\mathcal{P}^i,\label{Mni} 
\end{equation}

i.e.,  the matrix $M_{n,i}=(m_{j,l})$ is given by
\[m_{j,l} = \begin{cases}
-2 & \text{ if }j=l;\\
1 & \text{ if } |j-l| = i;\\
0 & \text{  otherwise, }  
\end{cases}\]
with the convention that the rows and columns of the matrix are indexed from $0$ to $n-1$, i.e., $0\leq j,l \leq n-1$.

In order to compute the number of solutions of the quadratic form defined by matrix $A$, we will consider an equivalent matrix obtained by using the matrix $M_{n,i}$. We will then compute the character of the determinant of this equivalent matrix. We begin by considering the case $i=1$.

\subsection{The case $i=1$.}

The following proposition provides us  a method to choose a suitable basis $\mathcal B$ that simplifies the calculation of $\chi(\det(A'))$, where $A'$ is a reduced matrix of $A$.

\begin{proposition}\label{beta}
Suppose that $\gcd(n,p)=p$. Then there exists an element $\beta \in \fqn\setminus \fq$ such that
\[\beta^{q}+\beta^{q^{n-1}}-2\beta = 0.\]
\end{proposition}
\begin{proof}
It is enough to show that the dimension of the subspace generated by the roots of the polynomial
\[\gcd(x^{q^n}-x, x^q+x^{q^{n-1}}-2x)\]
is at least two. It suffices to show that
\[\deg(\gcd(x^n-1, x+x^{n-1}-2)\ge 2.\]
Let us express $n$ as $n = p^u\tilde n$, where $\gcd(\tilde n, p) = 1$ and $u\ge 1$. Then we have
\begin{align*}
\deg(\gcd(x^n-1, x+x^{n-1}-2) & = \deg(\gcd(x^n-1, x^2-2x+1))\\
& = \deg((x^{\tilde n} -1)^{p^u}, (x-1)^{2})\\
& =\deg( (x-1)^{\min\{p^u, 2\}})= 2. 
\end{align*} 
This proves the proposition.
\end{proof}
According to Proposition \ref{beta}, we can choose the basis $\mathcal B$ in such a way that $\beta_n = 1$ and $\beta_{n-1} = \beta$, where $\beta$ satisfies
$$\beta^q+\beta^{q^{n-1}}-2\beta = 0,$$
which is equivalent to
\begin{equation}\label{beta1} \beta^{q^2}-2\beta^q+\beta=0.
\end{equation}
The above equation implies
\begin{equation}\label{beta2}
(\beta^q-\beta)^q = \beta^q-\beta.
\end{equation}
In particular, we have $\beta^q - \beta \in \mathbb{F}_q$. From now on, we fix the basis $\mathcal B = \{\beta_1, \beta_2, \dots, \beta_{n-2}, \beta, 1\}$. The following proposition determines the rank of $M_{n,1}$ and the determinant of one of its reduced matrices.

\begin{proposition}\label{det}
The rank  of the $n\times n$ matrix $M_{n,1} = \mathcal P^T-2 Id +\mathcal P$ over $\F_{q}$ is given by
\[\text{rank } M_{n,1} = \left\{
\begin{array}{ll}
n-1 \quad \text{ if } \gcd(n,p) = 1;\\
n-2  \quad \text{ if } \gcd(n,p) = p.
\end{array}\right.\]

In addition,  if $M_{n,1}'$ denotes the principal submatrix of $M_{n,1}$ constructed from the first rank$(M_{n,1})$ rows and columns, then $M_{n,1}'$ is a reduced matrix of $M_{n,1}$ and 
\[ \det M_{n,1}' = \left\{
\begin{array}{ll}
 (-1)^{n-1}n \quad &\text{ if } \gcd(n,p) = 1;\\
(-1)^{n-1}  \quad& \text{ if } \gcd(n,p) = p.
\end{array}\right.  \]
\end{proposition}

\begin{proof}
Let us denote $M_n$ as the matrix:
$$M_n= \begin{psmallmatrix}
-2 & 1 & 0 &   \ldots   & 0 & 0 & 0\\
1& -2&1& \ldots & 0 & 0& 0\\
0&1&-2& \ldots & 0 & 0&0\\
\vdots & \vdots &   \vdots &\ddots& \vdots&\ddots & \vdots \\
0& 0& 0 &  \ldots & -2&1 &0\\
0& 0& 0 &  \ldots & 1&-2 &1\\
0 & 0 & 0 &  \ldots &0&1&-2\\
\end{psmallmatrix}.$$
We observe that $M_{n,1} = M_n + R_n$, where $R_n = (r_{i,j})$ with
$r_{i,j} = \begin{cases}
1 & \text{ if } (i,j) \in\{(1,n),(n,1)\};\\
0 & \text{ otherwise. }
\end{cases}$
If we define  $$\displaystyle U=\begin{psmallmatrix}
1 & 0 & 0 &   \ldots   & 0 & 0 \\
0& 1&0& \ldots & 0 & 0\\
\vdots & \vdots &   \vdots &\ddots& \vdots & \vdots \\
0& 0& 0 &  \ldots & 1&0 \\
1 & 1 & 1 &  \ldots &1&1\\
\end{psmallmatrix}$$
 then we have  
 $$UM_{n,1}U^T=\begin{psmallmatrix}
-2 & 1 & 0 &  \ldots   & 0 & 0 & 0\\
1& -2&1& \ldots & 0 & 0& 0\\
0&1&-2& \ldots & 0 & 0&0\\
\vdots &  \ddots & \vdots &\ddots&\vdots&\ddots & \vdots \\
0& 0& 0 &  \ldots & -2&1 &0\\
0& 0& 0 &  \ldots & 1&-2 &0\\
0 & 0 & 0 &   \ldots&0&0&0 \\
\end{psmallmatrix} = \left(
\begin{array}{c|c}
M_{n-1}  & 0 \\ \hline
 0 & 0
\end{array}\right).$$
We claim that $L_{n-1} = \det(M_{n-1}) = (-1)^{n-1}n$ for $n > 1$.
To prove this, we expand the determinant of $M_{n-1}$ along the first row and obtain the recursive relation 
\begin{equation}\label{mh}
 L_{n-1}   =  -2  L_{n-2}- L_{n-3} \text{ for all } n\ge 4.
\end{equation}
This implies that the sequence ${L_n}_{\{n\geq2\}}$ satisfies the recurrence relation \eqref{mh} with the associated characteristic polynomial given by $(z+1)^2$. Since $-1$ is a double root of the characteristic polynomial, we have $L_{n-1} = a(-1)^n + b(-1)^n n$, where $a$ and $b$ are elements of $\mathbb{F}_q$. Considering that $L_2 = 3$ and $L_3 = -4$, we can conclude that $a = 0$ and $b = -1$. Thus, $L_{n-1} = (-1)^{n-1}n$ as desired.

Furthermore, if $\gcd(n,p) = 1$, it follows that $L_{n-1} = (-1)^{n-1}n \neq 0$, which means that the rank of $M_{n-1}$ is $n-1$. This implies that the rank of $M_{n,1}$ is also $n-1$.

In the case $\gcd(n,p) = p$, we define the matrix $V$ as

$$\displaystyle V=\begin{psmallmatrix}
1 & 0 & 0 &   \ldots   & 0& 0 & 0 & 0 \\
0& 1&0& \ldots & 0& 0 & 0& 0\\
\vdots & \vdots &   \vdots &\ddots& \vdots& \vdots & \vdots & \vdots \\
0& 0& 0 &  \ldots &1& 0&0&0 \\
0& 0& 0 &  \ldots &0& 1&0&0 \\
1 & 2 & 3 &  \ldots&n-3 & \frac {n-4}2&-1&0\\
0 & 0 & 0 &  \ldots &0&0&0&1\\
\end{psmallmatrix}.$$ 
We observe that $V$ is invertible, and by matrix operations, we have
$$VUM_{n,1}U^TV^T
=\left(
\begin{array}{c|c}
M_{n-2}  & 0\quad 0 \\ \hline
 \begin{array}{c}0 \cdots 0\\ 0 \cdots 0\end{array}& \begin{array}{cc}
 n&0\\
 0&0
 \end{array}
\end{array}\right).$$
Therefore, $L_{n-2} = (-1)^{n-2}(n-1) \neq 0$, and the rank of $M_{n,1}$ is $n-2$. Thus, the result follows.

\end{proof}
Using the notation in this proposition, we can express $A$ as follows

\begin{align*} A = \begin{cases} \frac{1}2 BU^{-1} \left(\begin{array}{c|c}
M_{n-1}  & 0 \\ \hline
 0 & 0
\end{array}\right)(U^{-1})^TB^T& \text{ if } \gcd(n,p) =1; \\
\frac{1}2BU^{-1}V^{-1}\left(
\begin{array}{c|c}
M_{n-2}  & 0\quad 0 \\ \hline
 \begin{array}{c}0 \cdots 0\\ 0 \cdots 0\end{array}& \begin{array}{cc}
 n&0\\
 0&0
 \end{array}
\end{array}\right) (V^{-1})^T(U^{-1})^TB^T& \text{ if } \gcd(n,p) =p.\\
\end{cases}
\end{align*}
Now, we will determine the determinant of a reduced matrix of $A$. This requires the use of the following proposition, which tell us about the quadratic character of the determinant of $BB^T$.
\begin{proposition}\label{qq}
Let $\mathcal B = \{\beta_1, \beta_2, \dots, \beta_{n-2}, \beta,1\}$ be a basis of $\fqn$ over $\fq$ and $B$ the matrix defined in \eqref{B}.
Then 
\[\chi(\det(BB^T)) = \begin{cases}
1 & \text{ if } n \text{ is odd;} \\
-1 & \text{ if } n \text{ is even.}
\end{cases}\]

\end{proposition}

\begin{proof}
We want to determine the determinant of $BB^T$. It is important to note that while $BB^T \in GL(\mathbb{F}_q)$, the individual matrices $B$ and $B^T$ do not necessarily belong to $GL(\mathbb{F}_q)$. Hence, we cannot directly compare the quadratic character $\chi(\det(BB^T))$ with $\chi(\det(B)^2)$ since the quadratic character $\chi$ is defined over $\mathbb{F}_q$. Therefore, we cannot conclude that $\chi(\det(BB^T))$ is equal to $\chi(\det(B)^2)$. The matrix $BB^T$ can be expressed as
\[BB^T = \begin{pmatrix}
\Tr(\beta_1^2) & \Tr(\beta_1\beta_2) & \cdots& \Tr(\beta_1\beta_n)\\
\Tr(\beta_2\beta_1) & \Tr(\beta_2^2) & \cdots &\Tr(\beta_2\beta_n)\\
\vdots & \cdots & \ddots & \vdots\\
\Tr(\beta_n\beta_1) & \Tr(\beta_n\beta_2) & \cdots& \Tr(\beta_n^2)\\
\end{pmatrix}\in GL_n(\F_q).\]
Let $\sigma: (\overline\F_q)\to \fq $ be the Frobenius map, and let $B= (c_{i,j})$ with $0\le i,j\le n-1$, where $c_{i,j} = c_i^{q^j}$ and $c_i= c_{i0}= \beta_{i-1}$.
We have
\begin{align*}
\det(B) &= \sum_{\tau \in S_n} (-1)^{\text{sing}(\tau)}c_{0\tau(0)}\dots c_{n-1\tau(n-1)}\\
&= \sum_{\tau \in S_n} (-1)^{\text{sing}(\tau)}c_{0}^{q^{\tau(0)}}\dots c_{n-1}^{q^{\tau(n-1)}}\\
\end{align*}
Then 
\begin{align*}
\sigma(\det(B)) &=  \sum_{\tau \in S_n} (-1)^{\text{sing}(\tau)}c_{0}^{q^{\tau(0)+1}}\dots c_{n-1}^{q^{\tau(n-1)+1}}\\
& = \sum_{\tau \in S_n} (-1)^{\text{sing}(\tau)}c_{0}^{\rho(\tau(0))}\dots c_{n-1}^{\rho(\tau(n-1))}\\
& = (-1)^{\text{sing}(\rho)}\det(B)\\
\end{align*}
where $\rho$ is the permutation $\begin{pmatrix}
0&1&2& \cdots & n-1
\end{pmatrix}$.
Since $$\begin{pmatrix}
0&1&2& \cdots & n-1
\end{pmatrix}= (0 \, n-1) (0 \, n-2) \cdots (0 \, 2) (0 \, 1),$$ we have that $(-1)^{\text{sing}(\rho)}= (-1)^{n-1}.$
Therefore, $\det(B) \in \fq$ if and only if $n$ is odd. In conclusion
\[\chi(\det(BB^T)) = \begin{cases}
1 & \text{ if } n \text{ is odd;} \\
-1 & \text{ if } n \text{ is even.}
\end{cases}\]
\end{proof}

\begin{proposition}\label{detB}
The rank  of the $n\times n$ matrix $A$ over $\F_{q}$ is given by
\[\text{rank } A = \left\{
\begin{array}{ll}
n-1 \quad \text{ if } \gcd(n,p) = 1;\\
n-2  \quad \text{ if } \gcd(n,p) = p,
\end{array}\right.\]
and a reduced matrix of $A$ is given by $A' = \begin{cases}
B_{11}M_{n-1}B_{11}^T & \text{ if } \gcd(n,p) =1; \\
\tilde B_{11}M_{n-2}\tilde B_{11}^T & \text{ if } \gcd(n,p) =p, \\
\end{cases}$ where the matrices $B_{11} = (b_{l,k})$ and   $ \tilde B_{11} = (\tilde b_{l,k})$ are given by
$$b_{l,k} = \beta_l^{q^{k-1}}- \beta_l^{q^{n-1}} \quad \text{ for }\quad  0\le l,k\le n-1$$
 and 
  \[b'_{l,k}=  \begin{cases}
\beta_l^{q^{k-1}}+k\beta_l^{q^{n-2}}-(k+1)\beta_l^{q^{n-1}} & \text{ if } 1\le l \le n-2,1\le k\le n-3;\\
\beta_l^{q^{k-1}}+\frac{n-4}2\beta_l^{q^{n-2}}-\frac{n-2}2\beta_l^{q^{n-1}} & \text{ if } 1\le l \le n-2,k=n-2.\\
\end{cases}\]
 In addition,  
\[ \chi(\det A') = \left\{
\begin{array}{ll}
(-1)^{n-1} \cdot \chi( (-2)^{n-1}n) \quad &\text{ if } \gcd(n,p) = 1;\\
(-1)^{n-1}\chi\left((-1)^{n-1}2^{n-2})\right) \quad& \text{ if } \gcd(n,p) = p.
\end{array}\right.  \]
\end{proposition}

\begin{proof} We will divide the proof into two cases, based on the value of $\gcd(n,p)$, using the notation introduced in the proof of Proposition \ref{det} for the matrices $M_n$, $U$, and $V$.
\begin{enumerate}
\item If $\gcd(n,p)=1$. In this case, we have that 
\[A =\frac{1}2 BU^{-1} \left(\begin{array}{c|c}
M_{n-1}  & 0 \\ \hline
 0 & 0
\end{array}\right)(U^{-1})^TB^T\]
and
\begin{align} \label{AA}
BU^{-1} \nonumber
& = \begin{pmatrix}
\beta_1-\beta_1^{q^{n-1}}& \beta_1^q-\beta_1^{q^{n-1}}& \cdots& \beta_1^{q^{n-2}}-\beta_1^{q^{n-1}}& \beta_1^{q^{n-1}}\\
\beta_2-\beta_2^{q^{n-1}}& \beta_2^q-\beta_2^{q^{n-1}}& \cdots &\beta_2^{q^{n-2}}-\beta_2^{q^{n-1}}& \beta_2^{q^{n-1}}\\
\vdots & \cdots& \ddots &\cdots & \vdots\\
\beta-\beta^{q^{n-1}}& \beta^q-\beta^{q^{n-1}}& \cdots& \beta^{q^{n-2}}-\beta^{q^{n-1}}& \beta^{q^{n-1}}\\
0& 0& \cdots &0& 1
\end{pmatrix}\\
& = \left(\begin{array}{cc}
B_{11}& B_{12}\\
B_{21} & 1
\end{array}\right),
\end{align}
where $B_{12} = \begin{pmatrix}
\beta_1^{q^{n-1}}\\
\beta_2^{q^{n-1}}\\
\vdots\\
\beta^{q^{n-1}}\\
\end{pmatrix}$ and $B_{21} =\begin{pmatrix}
0&0&\cdots&0 \end{pmatrix}$. From \eqref{AA}, it follows that $\det(BU^{-1}) = \det(B_{11})$ and we have that
\begin{align}\label{pq}
\chi\left(\det(B_{11}M_{n-1}B_{11}^T)\right) & =\nonumber \chi\left(\det(BU^{-1}M_{n-1}(U^{-1})^T
B^T)\right)\\ \nonumber
& = \chi(\det(BB^T)\det(M_{n-1}))\\
& =(-1)^{n-1} \cdot \chi( (-1)^{n-1}n)\\ \nonumber
\end{align}
Since 
\begin{align*}
A & = \frac{1}2  \left(\begin{array}{cc}
B_{11}& B_{12}\\
B_{21} & \beta_n^{q^{n-1}}
\end{array}\right)  \left(\begin{array}{cc}
M_{n-1}& 0\\
0 & 0
\end{array}\right) \left(\begin{array}{cc}
B_{11}^T& B_{21}^T\\
B_{12}^T & \beta_n^{q^{n-1}}
\end{array}\right)\\
& = \frac{1}2 \left(\begin{array}{cc}
B_{11}M_{n-1}B_{11}^T& 0\\0&0
\end{array}\right),
\end{align*}
it follows that  $A$ has $B_{11}M_{n-1}B_{11}^T$ as a reduced matrix.  Additionally, we have 
$$\chi(\det(B_{11}M_{n-1}B_{11}^T)= (-1)^{n-1} \cdot \chi( (-2)^{n-1}n),$$
where the last step follows from \eqref{pq}.
\item 
If $\gcd(n,p)=p$, we have the expression for the matrix $A$ as follows
\begin{align*}
A & = \frac{1}2BU^{-1}V^{-1}\left(
\begin{array}{c|c}
M_{n-2}  & 0\quad 0 \\ \hline
 \begin{array}{c}0 \cdots 0\\ 0 \cdots 0\end{array}& \begin{array}{cc}
 n&0\\
 0&0
 \end{array}
\end{array}\right) (V^{-1})^T(U^{-1})^TB^T,\\
\end{align*}
where the matrices $U^{-1}V^{-1}$ and $BU^{-1}V^{-1}$ are given by
\[U^{-1}V^{-1} = \begin{psmallmatrix}
1 & 0 & 0 &   \ldots   & 0& 0 & 0 & 0 \\
0& 1&0& \ldots & 0& 0 & 0& 0\\
\vdots & \vdots &   \vdots &\ddots& \vdots& \vdots & \vdots & \vdots \\
0& 0& 0 &  \ldots &1& 0&0&0 \\
0& 0& 0 &  \ldots &0& 1&0&0 \\
1 & 2 & 3 &  \ldots&n-3 & \frac {n-4}2&-1&0\\
-2 &-3 & -4 &  \ldots &-(n-2) &\frac{n-2}2&1&1\\
\end{psmallmatrix} \]
and $BU^{-1}V^{-1} = (b_{l,k}) $ where 
\[ b_{l,k}=  \begin{cases}
\beta_l^{q^{k-1}}+k\beta_l^{q^{n-2}}-(k+1)\beta_l^{q^{n-1}} & \text{ if } 1\le l \le n,1\le k\le n-3;\\
\beta_l^{q^{k-1}}+\frac{n-4}2\beta_l^{q^{n-2}}-\frac{n-2}2\beta_l^{q^{n-1}} & \text{ if } 1\le l \le n,k=n-2;\\
-\beta_l^{q^{n-2}}+\beta_l^{q^{n-1}} & \text{ if } 1\le l \le n,k=n-1;\\
\beta_l^{q^{n-1}} & \text{ if } 1\le l \le n,k=n.\\
\end{cases}\]
We observe that the only non-zero entry in the last row is $b_{n,n}=1$. In order to determine the entries in the $(n-1)$-th row, we need the following result.

\textbf{Claim}: $\beta^{q^{l}}= l \beta^q-(l-1)\beta$, for $l\ge2$.

We will prove this by induction. For $l=2$, the result follows from  \eqref{beta1}. Let us assume that the claim is true for some $k\ge 2$. Then, for $l=k+1$, we have:
\begin{align*}
\beta^{q^{k+1}}  = (\beta^{q^k})^q &= (k \beta^q-(k-1)\beta)^q\\
&= k \beta^{q^2} - (k-1)\beta^q\\
& = 2k \beta^q-k\beta -(k-1)\beta^q\\
& = (k+1)\beta^q-k \beta,
\end{align*}
Thus, the claim holds for $l=k+1$, completing the induction step and proving the claim.

From this relation, we can prove  that the only entries  non-null of the $n-1$-th row of the matrix $BU^{-1}V^{-1}$ are $b_{n-1,n-1}$ and $b_{n,n-1}$. In fact,  using that  $\gcd(n,p)=p$,  we have that
$$\beta^{q^{n-2}}=(n-2)\beta^q-(n-3)\beta=-2\beta^q+3\beta$$
and
$$\beta^{q^{n-1}}=(n-1)\beta^q-(n-2)\beta=-\beta^q+2\beta.$$

Therefore  the entries $b_{n-1,l}$ for $1\le l \le n-3 $ are given  by

\begin{align*}
b_{n-1,l}& = \beta^{q^{l-1}}+l\beta^{q^{n-2}}-(l+1)\beta^{q^{n-1}} \\
& = (l-1)\beta^q-(l-2)\beta+l(-2\beta^q+3\beta)-(l+1)(-\beta^q+2\beta)\\
& = 0,\\
\end{align*}
the entry $b_{n-1,n-2}$ is 
{\small \begin{align*}
b_{n-1,n-2} & =  \beta^{q^{n-3}}+\frac{n-4}2\beta^{q^{n-2}}-\frac{n-2}2\beta^{q^{n-1}}\\
& =  \beta^{q^{n-3}}-2\beta^{q^{n-2}}+\beta^{q^{n-1}}\\
& = (n-3)\beta^q-(n-4)\beta -2(-2\beta^q+3\beta)-\beta^q+2\beta\\
& =0
\end{align*}}
and the entry $b_{n-1,n-2}$ is 
\begin{align*}
b_{n-1,n-1} & = -\beta^{q^{n-2}}+\beta^{q^{n-1}}\\
& = 2\beta^q-3\beta -\beta^q+2\beta\\
& = \beta^q-\beta.
\end{align*}

Consequently 
\[b_{n-1,k} = \begin{cases}
0 & \text{ if } 1\le k \le n-2;\\
\beta^q-\beta & \text{ if }  k=n-1;\\
\beta & \text{ if } k=n,
\end{cases}\]
and we obtain that 
{\small
\begin{align*}
A & = \frac{1}2BU^{-1}V^{-1}\left(
\begin{array}{c|c}
M_{n-2}  & 0\quad 0 \\ \hline
 \begin{array}{c}0 \cdots 0\\ 0 \cdots 0\end{array}& \begin{array}{cc}
 n&0\\
 0&0
 \end{array}
\end{array}\right) (V^{-1})^T(U^{-1})^TB^T\\
& = \frac{1}2 \left(\begin{array}{c|c}
\tilde B_{11} &\tilde B_{12}\\ \hline
\tilde B_{21} &\begin{array}{cc}
\beta^q-\beta& \beta\\
0&1\\
\end{array}
\end{array} \right)\left(
\begin{array}{c|c}
M_{n-2}  & 0\quad 0 \\ \hline
 \begin{array}{c}0 \cdots 0\\ 0 \cdots 0\end{array}& \begin{array}{cc}
 n&0\\
 0&0
 \end{array}
\end{array}\right) \left(\begin{array}{c|c}
\tilde B_{11}^T &\tilde B_{21}^T\\ \hline
\tilde B_{12}^T &\begin{array}{cc}
\beta^q-\beta& \beta\\
0&1\\
\end{array}
\end{array} \right)\\
& = \frac{1}2  \left(\begin{array}{c|c}
\tilde B_{11}M_{n-2}\tilde B_{11}^T&  \begin{array}{cc}0&0 \\ \vdots &\vdots\\ 0& 0\end{array}\\\hline 
 \begin{array}{c}0 \cdots 0\\ 0 \cdots 0\end{array} & \begin{array}{cc}
0&0\\
0&0\\
\end{array}
\end{array}\right),
\end{align*}}
where $\tilde B_{12} =\left( \begin{array}{ccc}
\beta_1^{q^{n-1}}-\beta_1^{q^{n-2}}& \beta_1\\
\vdots& \vdots\\
\beta_{n-2}^{q^{n-1}}-\beta_{n-2}^{q^{n-2}}& \beta_{n-2}
\end{array}\right)$ and $\tilde B_{21} = \left(\begin{array}{cccc}
0&0 & \cdots & 0\\
0&0 & \cdots & 0\\
\end{array}\right)$.
This implies that

\begin{align*}
\chi(\det(\tilde B_{11} ))& =\chi\left( \frac1{\beta^q-\beta}\det\left( \begin{array}{c|c}
\tilde B_{11} & \quad *\qquad *\\ \hline
\begin{array}{cc}0\\ 0\end{array}& \begin{array}{cc}
\beta^q-\beta& \beta\\
0&1\\
\end{array}
\end{array}\right)\right)\\
&= \chi\left(\frac{1}{\beta^q-\beta} \det(BU^{-1}V^{-1})\right)\\
&  = \chi\left(\frac{1}{\beta^q-\beta} \det(B)\right). \\
\end{align*}
\end{enumerate}
From the previous analysis, we can conclude that a reduced matrix of $A$ is given by $\tilde{B}{11}M{n-2}\tilde{B}_{11}^T$, which has a rank of $n-2$. We can now evaluate the quadratic character of the determinant of this matrix, i.e., 

\begin{align*}
\chi\left(\det \left(\tilde B_{11}M_{n-2} \tilde B_{11}^T\right)\right)& =\chi\left( \det\left(\left(\frac{1}{\beta^q-\beta}\right)BU^{-1}V^{-1}M_{n-2}\left(\frac{1}{\beta^q-\beta}\right)(V^{-1})^T(U^{-1})^TB^T)\right)\right) \\
& = \chi\left(\left(\frac{1}{\beta^q-\beta}\right)^2\det(BB^T)\det(M_{n-2})\right)\\
& = (-1)^{n-1}\chi\left((-1)^{n-1})\right),\\
\end{align*} 
where we use Propositions \ref{det} and \ref{qq}. Therefore, the proposition is proven.
\end{proof}

The following definition will be useful to determine $N_n(C_i)$.

\begin{definition}
For each  $\alpha \in \F_q$ we define
\[\varepsilon_{\alpha} = \begin{cases}
q-1 & \text{ if } \alpha =0; \\
-1 & \text{ otherwise.} \\
\end{cases}
\]
\end{definition}

From Theorem \ref{sol}  and  Proposition \ref{detB} we have the following theorem.
\begin{theorem}\label{m1}
Let $\lambda \in \fqn$ and $n$ a positive integer. The number $N_n(\mathcal C_{1})$ of affine rational points   in $\F_{q^n}^2$  of the curve determined by the equation  $y^q-y=x^{q+1}-x^2-\lambda$ is  
\[N_n(\mathcal C_1)= \left\{
\begin{array}{llll}
 q^{n}-  \chi(2(-1)^{\frac{n}2}n\Tr(\lambda))q^{(n+2)/2} \quad & \text{ if } \gcd(n,p) = 1 \text{ and } n \text{ is even;}\\
 q^{n}+ \varepsilon_{\Tr(\lambda)} \chi((-1)^{(n-1)/2}n)q^{(n+1)/2}  \quad &\text{ if } \gcd(n,p) = 1 \text{ and } n \text{ is odd;}\\
q^{n}-\varepsilon_{\Tr(\lambda)} \chi((-1)^{n/2}) q^{(n+2)/2} \quad &\text{ if } \gcd(n,p) = p \text{ and } n\text{ is even;}\\
q^{n}  +\chi(2(-1)^{\frac{n-3}2}\Tr(\lambda))q^{(n+3)/2}\quad &\text{ if } \gcd(n,p) = p \text{ and } n \text{ is odd.}
\end{array}\right.\]
\end{theorem}
This theorem allows us to determine when $\mathcal C_{1}$ is minimal or maximal with respect  the Hasse-Weil bound,  as we show in  the following corollary. 
\begin{theorem}
Consider the curve  $\mathcal C_1$ given by
\[\mathcal C_1 : y^q -y = x^{q+1}-x^2-\lambda.\]
Then $\mathcal C_1$ is minimal  in $\F_{q^n}$ if and only if one of the following holds
\begin{enumerate}[a)]
\item $\Tr(\lambda) =0$,  $2p$ divides $n$ and $q\equiv 1 \pmod 4$; 
\item $\Tr(\lambda) =0$,  $4p$ divides $n$ and $q\equiv 3 \pmod 4$.
\end{enumerate}
Moreover, $\mathcal C_1$ is maximal in $\F_{q^n}$ if and only if $\Tr(\lambda) =0,$ $2p$ divides $n$, $4$ does not divide $n$  and $q \equiv 3 \pmod 4.$
\end{theorem}
\begin{proof}
The result follows from Theorem \ref{m1} and the fact that the genus of $\mathcal C_1$ is $g= \frac{q(q-1)}2$.
\end{proof}

\subsection{The curve $y^q-y=x(x^{q^i}-x)-\lambda$ with $i\ge1$.}

Now let us consider the case $i \geq 1$. The following proposition provides information about the rank of the matrices $M_{n,i}$ and the quadratic character of the determinant of one of its reduced matrices.

\begin{proposition}\label{det2}
Let $i,n$ be integers such that $0<i< n$ and $M_{n,i}$ as defined in \eqref{Mni}. Set  $d= \gcd(i,n)$ and $ l = \frac{n}d$. The rank  of the $n\times n$ matrix $M_{n,i}$ is $n-d$ if $n=2i$ and, otherwise we have that 
\[\text{rank } M_{n,i} = \left\{
\begin{array}{ll}
n-d\quad &\text{ if } \gcd(l,p) = 1;\\
 n-2d \quad &\text{ if } \gcd(l,p) = p.
\end{array}\right.\]

In addition, the matrices
\begin{equation}\label{diagonal}
\tilde M_{n,i}=\left(
\begin{array}{c|c|c|c}
M_{l,1}  & 0&0 & 0 \\ \hline
 0& M_{l,1}&0&0\\ \hline 
 \vdots & \cdots & \ddots& \vdots\\ \hline
 0&0&0& M_{l,1}\\ 
 \end{array}\right) \quad \text{ and } \quad 
 \tilde M'_{n,i}=\left(
\begin{array}{c|c|c|c}
M'_{l,1}  & 0&0 & 0 \\ \hline
 0& M'_{l,1}&0&0\\ \hline 
 \vdots & \cdots & \ddots& \vdots\\ \hline
 0&0&0& M'_{l,1}\\ 
 \end{array}\right) 
 \end{equation}
 are an equivalent matrix and a reduced matrix of $M_{n,i}$, respectively, where $\tilde M'_{l,1}$ is as the matrix given in Proposition \ref{det}. 

The determinant of the matrix $\tilde M_{n,i}$ is $(-1)^{i}2^i$ if $n=2i$ and, otherwise we have that 
\[ \det \tilde M'_{n,i} = \left\{
\begin{array}{ll}
 (-1)^{n-d}l^d \quad &\text{ if } \gcd(l,p) = 1;\\
 (-1)^{n} \quad& \text{ if } \gcd(l,p) = p.
\end{array}\right.  \]
\end{proposition}
\begin{proof}
For convenience, let us enumerate the rows and columns of the matrix $M_{n,i}$ from $0$ to $n-1$. Suppose that $n$ is even and $i = \frac{n}{2}$. In this case, the matrix $M_{n,i}$ can be expressed as follows
\[a_{k,l} = \begin{cases}
-2 &\text{ if } k=l,\\
2 &\text{ if } k-l \equiv 0 \pmod {\frac{n}2},\\
0 &\text{ otherwise. }
\end{cases}\]
Let us denote $D_{\frac{n}{2}} = 2 \cdot \text{Id}_{\frac{n}{2}}$, where $\text{Id}_{\frac{n}{2}}$ is the $\frac{n}{2} \times \frac{n}{2}$ identity matrix. It follows that

\[M_{n,n/2} = \left(\begin{array}{c|c}
-D_{n/2}&D_{n/2}\\ \hline
D_{n/2}& -D_{n/2} \\ 
\end{array}\right) \text{ that is equivalent to } \left(\begin{array}{c|c}
-D_{n/2}&0\\ \hline
0& 0 \\ 
\end{array}\right).\]
Therefore
\[\text{ rank } M_{n,\frac{n}2} = \frac{n}2=i \quad \text{ and } \quad \text{ det } \tilde M'_{n, \frac{n}2} = (-2)^{\frac{n}2}=(-2)^i\neq 0.\]

This completes the proof for the case $n=2i$. For the remaining cases, we will first establish that it suffices to consider the case where $i=d$. Subsequently, we will construct a block diagonal matrix consisting of $d$ matrices of the form $M_{l,1}$, where $n=ld$.

We observe that any permutation $\rho:\mathbb{Z}_n \rightarrow \mathbb{Z}_n$ induces a natural action on $\mathbb{F}_{q^n}$ through the following map:

$$\begin{matrix}
\rho: & \fq^n &\to& \fq^n\\
& (v_0, \dots, v_{n-1})& \mapsto& (v_{\rho(0)}, \dots, v_{\rho(n-1)}).
\end{matrix}$$

This action is associated with an invertible matrix $M_{\rho}$ defined by

\[M_{\rho} \left(\begin{matrix}
v_0\\\vdots\\v_{n-1}
\end{matrix}\right) =\left(\begin{matrix}
v_{\rho(0)}\\\vdots\\v_{\rho(n-1)}
\end{matrix}\right) .\]

Conversely, for any permutation matrix $R$, there exists a permutation $\rho': \mathbb{Z}_n \rightarrow \mathbb{Z}_n$ such that $M_{\rho'} = R$. In particular, we observe that the map $\mathcal{P}^i$ determines the permutation

$$\begin{matrix}\pi_i:& \Z_n & \to & \Z_n \\
&z & \mapsto & z+i.
\end{matrix}$$

Consider the map
\begin{align*}
\sigma:  \Z_n &\to \Z_n\\
a & \mapsto ai
 \end{align*}
where $a \in \mathbb{Z}_n$ is an element satisfying $\gcd(a,n) = 1$. Since $a$ and $n$ are coprime, $\sigma$ is a permutation. Furthermore, $\sigma$ induces a matrix $M_{\sigma}$, and the matrix $M_{\sigma} \mathcal{P}^i M_{\sigma}^{-1}$ corresponds to a permutation of $\mathbb{Z}_n$ given by
\begin{align*}
\sigma \circ \pi_i \circ \sigma^{-1} (z) & = \sigma(\pi_i(\sigma^{-1}(z)))\\
& =  \sigma(\pi_i(a^{-1}z))\\
&=\sigma(a^{-1}z+i)\\
& = z+ai.
\end{align*}

We know that the congruence $ai \equiv u \pmod{n}$ has a solution if and only if $d$ divides $u$. Since $d = \gcd(i,n)$, there exists $a \in \mathbb{Z}_n$ such that $\sigma \circ \pi_i \circ \sigma^{-1}(z) = z + \gcd(i,n)$. Therefore, we can replace $i$ with $d$ without loss of generality.
For each $z \in \mathbb{Z}_n$, by the Euclidean Division Algorithm, there exist unique integers $r$ and $s$ such that $0 \leq s \leq l-1$ and $0 \leq r \leq d-1$ satisfying $z = sd + r$. Let us consider the map

\begin{equation}
\begin{matrix}
\varphi:&  \Z_n  &\to & \Z_n\\
&sd+r &\mapsto&  s+lr.
\end{matrix} 
\end{equation}

\textbf{Claim:} The map $\varphi$ is a permutation of the elements of $\mathbb{Z}_n$. Let us suppose that there exist distinct elements $z_1$ and $z_2$ in $\mathbb{Z}_n$ such that $\varphi(z_1) = \varphi(z_2)$. By the Euclidean Division, there exist $0 \leq s_1, s_2 \leq l-1$ and $0 \leq r_1, r_2 \leq d-1$ with $z_1 = s_1d + r_1$ and $z_2 = s_2d + r_2$. Then we have
\[\varphi(s_1d+r_1) =  \varphi(s_2d+r_2) \Leftrightarrow s_1+lr_1=s_2+lr_2 \Leftrightarrow s_1-s_2=l(r_2-r_1).\]
Since $0 \leq s_1, s_2 \leq l-1$, the above equation implies that $s_1 = s_2 = 0$ and $r_1 = r_2 = 0$. However, this contradicts the fact that $z_1 \neq z_2$. Therefore, $\varphi$ is a permutation.

We will utilize the permutation $\varphi$ to rearrange the rows and columns of $\mathcal{P}^d - 2\mathrm{Id} + \left(\mathcal{P}^d\right)^T$ and obtain the block diagonal matrix $\tilde{M}_{n,i}$ with $d$ blocks. Applying the permutation $\varphi$ to this matrix, we obtain the permuted matrix $\varphi(\mathcal{P}^d - 2\mathrm{Id} + \left(\mathcal{P}^d\right)^T)$.
This permuted matrix can be written as a block diagonal matrix with $d$ blocks, where each block is a cyclic permutation matrix of size $l \times l$:
\[ \tilde{M}_{n,i} = \begin{pmatrix}
M_{l,1} & 0 & \cdots & 0 \\
0 & M_{l,1} & \cdots & 0 \\
\vdots & \vdots & \ddots & \vdots \\
0 & 0 & \cdots & M_{l,1} \\
\end{pmatrix},\]
where $n = ld$. Thus, we have successfully rearranged the rows and columns to obtain the block diagonal matrix $\tilde{M}_{n,i}$ with $d$ blocks, where each block is a cyclic permutation matrix $M_{l,1}$ of size $l \times l$.

Let us observe that $\varphi \circ \pi_d \circ \varphi^{-1}$ defines a permutation $\theta: \Z_n \to \Z_n$ given by
\begin{align*}
\theta(z) = \varphi \circ \pi_d \circ \varphi^{-1}(z) & = \varphi \circ \pi_d (\varphi^{-1}(s+rl))\\
& = \varphi \circ \pi_d (sd+r)\\
& = \varphi ((s+1)d+r)\\
& = (s+1) + rl,
\end{align*} 
where $z=sl+r$ with $0\le s \le l-1$ and $0\le r \le d-1$. 
We  know that   $$M_{\varphi}\left(\bigl(\mathcal{P}^d\bigr)^{^{\ }_T} -2 Id +\mathcal P^d\right)M_{\varphi}^{-1}= \tilde M_{n,d}.$$ 
To demonstrate this, we show that the product of the permutation matrix $M_{\varphi}$ for $\left(\mathcal{P}^d\right)^{T} -2 \mathrm{Id} +\mathcal P^d$ and $M_{\varphi}^{-1}$ maps the nonzero entries of $M_{n,d}$ to the nonzero entries of $\tilde{M}_{n,d}$.
In the case when $k=j$, writing $k = s+rl$ with $0 \leq s \leq l-1$ and $0 \leq r \leq d-1$, we have
\[ \theta(a_{k,k}) = a_{(s+1)+rl,(s+1)+rl}=a_{\theta(k),\theta(k)},\]
which implies that the new matrix will have $a_{k,k} = -2$.
If $a_{k,k+d} = 1$, with $k = s+rl$ and $0 \leq s \leq l-1$, we have

\[ \theta(a_{k,k+d}) = a_{(s+1)+rl,(s+2)+rl} = a_{\theta(k),\theta(k+d)},\]
which implies that $a_{k,k+1} = 1$ when considering the indices modulo $l$. The same procedure applies to $a_{k+d,k}$, as it is the transpose of $a_{k,k+d}$. The other entries are zero and their images are also zero. Therefore, we obtain the matrix in  \eqref{diagonal}.

Using Proposition \ref{det} and the fact that the matrix $M_{n,i}$ is a block diagonal matrix with $d$ blocks equal to the matrix $M_{l,1}$, we can determine the rank of $M_{n,i}$ and the determinant of the reduced matrix $\tilde{M}'_{n,i}$. In conclusion, we obtain that 
 \[\text{ rank } M_{n,i} = \left\{\begin{array}{ll}
(l-1)d = n-d\quad \text{ if } \gcd(l,p) = 1;\\
 (l-2)d =n-2d \quad \text{ if } \gcd(l,p) = p,
\end{array}\right.\]
and 
\[ \det \tilde M_{n,i} = \left\{
\begin{array}{ll}
 \left((-1)^{l-1}l\right)^d  = (-1)^{n-d}l^d\quad &\text{ if } \gcd(l,p) = 1;\\
 \left((-1)^{l-1}\right)^d  = (-1)^{n-d} \quad& \text{ if } \gcd(l,p) = p.
\end{array}\right.  \]

\end{proof}
From Proposition \ref{pr2}, we know that the matrix associated with the quadratic form $\mathrm{Tr}(c x(x^{q^i}-x))$ is given by

\[A= \frac{c}2B(\p^T-2Id+\p)B^T\] 
where $B$ is defined in \eqref{B}. By applying Propositions \ref{detB} and \ref{det2}, we obtain the following result.

\begin{corollary}\label{b1}
Let $i$ be an integer such that $0<i<n$. Set $d=\gcd(i,n)$ and $l = \frac{n}d$. The rank of the $n\times n$ matrix $A= \frac{1}2B(\left(\mathcal{P}^i\right)^{^{\ }_T} -2 Id +\mathcal P^i)B^T$  is given by
\[\text{rank } A = \begin{cases}
n-d & \text{ if } \gcd(l,p)=1;\\
n-2d & \text{ if } \gcd(l,p)=p. 
\end{cases}\]
Let $A'$ be a reduced matrix of $A$. Then
\[\chi(\det(A')) =  \begin{cases}
(-1)^{n-d}\chi((-2)^{n-d}l^d) & \text{ if } \gcd(l,p)=1;\\
(-1)^{n-d}\chi((-1)^{n-d}2^{n-2d}) & \text{ if } \gcd(l,p)=p.
\end{cases}\]

\end{corollary}

By Theorem \ref{sol}  and  Propositions \ref{pr2} and \ref{det2} we have the following theorem.
\begin{theorem}\label{mm}
Let $n,i$ be integers such that $0<i<n$ and put $d= \gcd(i,n)$ and $ l = \frac{n}d$.  If $n=2i$, the number $N_n(\mathcal C_{i})$ of affine rational points   in $\F_{q^n}^2$  of the curve determined by the equation  $y^q-y=x^{q^i+1}-x^2-\lambda$ is  
\[N_n(\mathcal C_{i})= \left\{
\begin{array}{llll}
 q^{n}+  \chi(2(-1)^{(i+1)/2}\Tr(\lambda)) q^{(3i+1)/2}\quad & \text{ if } i \text{ is odd;}\\
 q^{n}+ \varepsilon_{\Tr(\lambda)} \chi((-1)^{i/2}) q^{3i/2} \quad &\text{ if } i  \text{  is even.}\\
\end{array}\right.\]
If $n\neq 2i$, the number of affine rational points of $\mathcal C_i$ is 
\[N_n(\mathcal C_{i})= \left\{
\begin{array}{llll}
 q^{n}- \chi(2(-1)^{(n-d+1)/2}\Tr(\lambda) l^d) q^{(n+d+1)/2} \quad & \text{ if } \gcd(l,p) = 1 \text{ and } n+d \text{ is odd;}\\
 q^{n}+ \varepsilon_{\Tr(\lambda)} \chi((-1)^{(n-d)/2}l^d) q^{(n+d)/2} \quad &\text{ if } \gcd(l,p) = 1  \text{ and $n+d$ is even;}\\
q^{n}-(-1)^d  \chi(2(-1)^{(n+1)/2}\Tr(\lambda)) q^{(n+2d+1)/2}  \quad &\text{ if } \gcd(l,p) = p \text{ and } n \text{ is odd;}\\
q^{n}+(-1)^d \varepsilon_{\Tr(\lambda)} \chi((-1)^{n/2})q^{(n+2d)/2} \quad &\text{ if } \gcd(l,p) = p \text{ and } n\text{ is even.}\\
\end{array}\right.\]
\end{theorem}

\begin{remark}
The curve 
$\mathcal C_{i}$ has genus $g = \frac{(q-1)q^i}2$.
The Hasse-Weil bound of $\mathcal C_{i}$ is given by 
\[|N_n(\mathcal C_{i})-q^n|\le (q-1)q^{\frac{n+2i}2}. \]
\end{remark} 
Using Theorem \ref{mm}, we can determine the conditions when the curve $\mathcal C_i$ is maximal (or minimal) with respect the Hasse-Weil bound. 
\begin{theorem}
Let $n,i$ be integers such that $0<i<n$, set $d= \gcd(i,n)$ and $ l = \frac{n}d$. The curve 
$$\mathcal C_{i}: y^q-y = x(x^{q^i}-x)-\lambda$$ is maximal in $\F_{q^n}$ if and only if $\Tr(\lambda)=0$, $2p$ divides $n$, $i$ divides $n$ and $(-1)^d\chi((-1)^{n/2})=1$.
The curve $\mathcal C_i$ is minimal in $\F_{q^n}$ if and only if $\Tr(\lambda)=0$, $2p$ divides $n$, $i$ divides $n$ and $(-1)^{d}\chi((-1)^{n/2})=-1$. 
\end{theorem}

\section{The number of affine rational points of the hypersufarce $y^q-y = \sum_{j=1}^r a_jx_j(x_j^{q^{i_j}}-x_j)-\lambda$}

Using the results derived in the previous section, we can now determine the number  of affine rational points of   Artin-Schreier's hypersurface $\mathcal H_r$  defined by the equation:
 
$$\mathcal H_r : y^q-y = \sum_{j=1}^r a_jx_j(x_j^{q^{i_j}}-x_j)-\lambda,$$
where $a_j \in \fq^*$ and $0<i_j<n$ for $j\in\{1,\dots,r\}$.  According to Theorem $5.4$ in \cite{LiNi}, we have the following result
 $$\sum_{c\in\fqn} \psi(uc) =\begin{cases}
0 & \text{ if } u \neq 0;\\
q^n & \text{ if } u = 0.
\end{cases}$$ 
Using this result we can compute the number $N_{n}(\mathcal{H}_r)$ as follows

 \begin{align}\label{AA1} \nonumber
 q^n N_{n}(\mathcal H_r) & = \sum_{c \in \fqn} \sum_{x_1 \in \fqn} \cdots \sum_{x_r\in \fqn} \sum_{y\in \fqn} \psi\left(c\left( \sum_{j=1}^r a_jx_j(x_j^{q^{i_j}}-x_j)-y^q+y-\lambda\right)\right)\\ \nonumber
& =q^{(r+1)n}  + \sum_{c \in \fqn^*} \sum_{x_1 \in \fqn} \cdots \sum_{x_r\in \fqn} \psi\left(c\left( \sum_{i=1}^r a_jx_j(x_j^{q^{i_j}}-x_j)-\lambda\right)\right) \sum_{y\in \fqn} \psi\left(c\left( -y^q+y\right)\right) \\ \nonumber
& = q^{(r+1)n}  + \sum_{c \in \fqn^*} \psi(-c\lambda) \prod_{j=1}^r  \sum_{x_j \in \fqn}\psi\left(c\left(  a_jx_j(x_j^{q^{i_j}}-x_j)\right)\right) \sum_{y\in \fqn} \psi\left(c\left( -y^q+y\right)\right) \\ \nonumber 
& = q^{(r+1)n}  + \sum_{c \in \fqn^*} \psi(-c\lambda) \prod_{j=1}^r \sum_{x_j \in \fqn}\psi\left(c\left( a_j x_j(x_j^{q^{i_j}}-x_j)\right)\right) \sum_{y\in \fqn} \psi\left(y\left( -c^{q^{n-1}}+c\right)\right).\\
 \end{align}

We observe that $$\displaystyle \sum_{y\in \fqn} \psi\left(y\left( -c^{q^{n-1}}+c\right)\right) = \begin{cases}
q^n&  \text{ if } c^{q^{n-1}}-c =0;\\
0 & \text{otherwise}.\\
\end{cases}$$
Since $c^{q^{n-1}}-c = 0$ if and only if $c \in \mathbb{F}_{q^{n-1}}$, we can conclude that the inner sum in \eqref{AA1} has non-null terms only if $c \in \mathbb{F}_q$. Therefore, we have that
\begin{align} \label{AA2} 
N_{n}(\mathcal H_r) = q^{rn} + \sum_{c \in \fq^*} \tilde\psi(-c\Tr(\lambda))\prod_{j=1}^r\sum_{x_j \in \fqn} \psi\left(ca\left(  x_j(x_j^{q^{i_j}}-x_j)\right)\right) .
\end{align} 
The following theorem gives  explicit formulas for $N_n(\mathcal H_r)$. 
\begin{theorem}\label{th1}
Let $i_1, \dots, i_r$ be positive integers such that $0<i_j<n.$ We define $d_j = \gcd(i_j, n)$ and $l_j = \frac{n}{d_j}$. Let $X = \{l_1, \dots, l_{\tilde{r}}\}$ and $Y =\{l_{\tilde{r} + 1}, \dots, l_r\}$, where  $\gcd(l_k, p) = 1$ for $l_k \in X$ and $\gcd(l_k, p) = p$ for $l_k \in Y$.

Let $\mathcal{H}_r$ be  the hypersurface given by the expression

$$ y^q-y = \sum_{j=1}^r a_jx_j(x_j^{q^{i_j}}-x_j)-\lambda,$$
where $\lambda \in \fqn$ and $a_j \in \fq$. Let us define

$$D_1=\sum_{j=1}^{\tilde r} d_j, \quad D_2=\sum_{j=\tilde r+1}^r d_j, \quad  L_1 = \prod_{j=1}^{\tilde r} l_j^{d_j}.\quad A_1=\prod_{j=1}^{\tilde r} a_j^{n-d_j} \quad \text{and } A_2=\prod_{\tilde r}^r a_j ^n.$$

 Let 
$A=A_1A_2$.  In the case $X = \varnothing$ we set $L_1=1$, $A_1=1$, 
and in the case $Y = \varnothing$ we set $A_2=1$.  The number $N_n(\mathcal{H}_r)$ of affine rational points of $\mathcal{H}_r$ in $\mathbb{F}_q^{r+1}$ is given by

\[N_{n}(\mathcal H_r) = \begin{cases}
q^{rn} + (-1)^{D_2}\tau^{s(nr-D_1-2D_2)} \chi((-1)^{D_2}AL_1)\varepsilon_{\Tr(\lambda)}q^{\frac{nr+D_1+2D_2}2}& \text{ if } nr-D_1 \text{ is  even};\\
q^{rn} +(-1)^{D_2+1}\tau^{s(nr-D_1-2D_2+1)} \chi(2(-1)^{D_2}AL_1\Tr(\lambda)) q^{\frac{nr+D_1+2D_2+1}2}  & \text{ if }  nr-D_1 \text{ is odd}.\\

\end{cases}\]
\end{theorem}
\begin{proof}
From  \eqref{AA2} we obtain that 
\begin{equation}\label{jp}
N_{n}(\mathcal H_r) = q^{rn} + \sum_{c \in \fq^*} \tilde\psi(-c\Tr(\lambda))\prod_{j=1}^r\sum_{x_j \in \fqn} \psi\left(ca_j\left(  x_j(x_j^{q^{i_j}}-x_j)\right)\right).
\end{equation} 
We can divide this product into two parts: one from $1$ to $\tilde{r}$ and the other from $\tilde{r}+1$ to $r$. Using Lemma \ref{soma} and Corollary \ref{b1}, we can obtain the following expressions

{\small
\begin{align*}
\prod_{j=1}^{\tilde r }\sum_{x_j \in \fqn} \psi\left(ca_j\left(  x_j(x_j^{q^{i_j}}-x_j)\right)\right) & =\prod_{j=1}^{\tilde r } (-1)^{(s+1)(n-d_j)}\tau^{s(n-d_j)}\left((-1)^{n-d_j}\chi((-2ca_j)^{n-d_j}l_j^{d_j}\right)q^{n-(n-d_j)/2}\\
& = (-1)^{s(n\tilde r-D_1)}\tau^{s(n \tilde r-D_1)}\chi((-2c)^{n \tilde r-D_1}A_1L_1)q^{\frac{n\tilde r +D_1}2}
\end{align*}}
and 
{\scriptsize
\begin{align*}
\prod_{j=\tilde r +1}^{r }\sum_{x_j \in \fqn} \psi\left(ca_j\left(  x_j(x_j^{q^{i_j}}-x_j)\right)\right) & =\prod_{j=\tilde r +1}^r (-1)^{(s+1)(n-2d_j)}\tau^{s(n-2d_j)}\left((-1)^{n-d_j}\chi((-1)^{n-d_j}(2ca_j)^{n-2d_j}\right)q^{n-(n-2d_j)/2}\\
& = (-1)^{sn(r-\tilde r)-D_2}\tau^{sn(r- \tilde r)-2sD_2}\chi((-1)^{n(r-\tilde r)-D_2}(2c)^{n(r-\tilde r)}A_2)q^{\frac{n(r-\tilde r)+2D_2}2}.
\end{align*}}
Using the given expressions in  \eqref{jp}, we  obtain the following result
{\small
\begin{align} \label{mr}
N_{n}(\mathcal H_r)- q^{rn} & = (-1)^{s(nr -D_1)-D_2}\tau^{s(nr-D_1-2D_2)} \sum_{c \in \fq^*} \tilde\psi(-c\Tr(\lambda))\chi((-1)^{D_2}(-2c)^{nr-D_1}AL_1)q^{\frac{nr+D_1+2D_2}2}
\end{align}
}
We will now proceed with the proof by considering two cases.

\begin{enumerate}
\item  If $nr-D_1$ is even, then $\chi((-2c)^{nr-D_1})=1$. Substituting this in  \eqref{mr}, we have

\begin{align*}
N_{n}(\mathcal H_r)- q^{rn} & = (-1)^{D_2}\tau^{s(nr-D_1-2D_2)} \chi((-1)^{D_2}AL_1)q^{\frac{nr+D_1+2D_2}2}\sum_{c \in \fq^*} \tilde\psi(-c\Tr(\lambda))\\
& =\begin{cases} (-1)^{D_2}\tau^{s(nr-D_1-2D_2)} \chi((-1)^{D_2}AL_1)q^{\frac{nr+D_1+2D_2}2}(q-1) & \text{ if } \Tr(\lambda) = 0\\
(-1)^{D_2}\tau^{s(nr-D_1-2D_2)} \chi((-1)^{D_2}AL_1^{D_1})q^{\frac{2nr-D_1-2D_2}2}(-1) & \text{ if } \Tr(\lambda) \neq 0
\end{cases}\\
& = (-1)^{D_2}\tau^{s(nr-D_1-2D_2)} \chi((-1)^{D_2}AL_1)\varepsilon_{\Tr(\lambda)}q^{\frac{nr+D_1+2D_2}2}.
\end{align*}

\item If $nr-D_1$ is odd, then $\chi((-2c)^{nr-D_1})=\chi(-2c)$. Substituting this in  \eqref{mr}, we have
{\footnotesize
\begin{align*}
N_{n}(\mathcal H_r)- q^{rn} & = (-1)^{s-D_2}\tau^{s(nr-D_1-2D_2)} \chi(2(-1)^{D_2+1}AL_1)q^{\frac{nr+D_1+2D_2}2}\sum_{c \in \fq^*} \tilde\psi(-c\Tr(\lambda))\chi(c)\\
& =\begin{cases} (-1)^{s-D_2}\tau^{s(nr-D_1-2D_2)} \chi(2(-1)^{D_2+1}AL_1)q^{\frac{nr+D_1+2D_2}2}\displaystyle\sum_{c\in \fq^*}\chi(c) & \text{ if } \Tr(\lambda) = 0\\
(-1)^{s-D_2}\tau^{s(nr-D_1-2D_2)} \chi(2(-1)^{D_2}AL_1\Tr(\lambda))q^{\frac{nr+D_1+2D_2}2}\displaystyle\sum_{c\in \fq^*} \tilde \psi(-c\Tr(\lambda))\chi(-c\Tr(\lambda))& \text{ if } \Tr(\lambda) \neq 0
\end{cases}\\
& =\begin{cases}0 & \text{ if } \Tr(\lambda) = 0;\\
(-1)^{s-D_2}\tau^{s(nr-D_1-2D_2)} \chi(2(-1)^{D_2}AL_1\Tr(\lambda))q^{\frac{nr+D_1+2D_2}2}G_1(\tilde \psi, \chi)& \text{ if } \Tr(\lambda) \neq 0.
\end{cases}\\
\end{align*}
}
Applying Theorem \ref{gaussdois}, which states that $G_1(\tilde \psi, \chi) =-(-\tau)^s\sqrt{q}$, we obtain the following expression
\[N_{n}(\mathcal H_r)- q^{rn}  = (-1)^{D_2+1}\tau^{s(nr-D_1-2D_2+1)} \chi(2(-1)^{D_2}AL_1\Tr(\lambda)) q^{\frac{nr+D_1+2D_2+1}2}.\]
\end{enumerate}

\end{proof}

The well-known Weil bound tells us that
\[\left|N_n(\mathcal H_r) - q^{nr}\right|\le (q-1) \prod_{j=1}^r q^{i_j} q^{nr/2} = (q-1)q^{\frac{rn+2I}2},\]
where $I = \sum_{j=1}^r i_j. $  By Theorem \ref{th1} this bound can be attend if and only if  in the case $nr-D_1$ is even and $\Tr(\lambda)=0$. Besides that, we would need that $2I = D_1+2D_2$, that only occurs if $D_1=0$ and $i_{j}= d_{j}$ for $j\in Y$.
Using this fact, we obtain the following result, that assures us when the hypersurface $\mathcal H_r$ is maximal or minimal.

\begin{theorem}\label{HW}
Let $\mathcal H_r: y^q-y = \displaystyle\sum_{j=1}^r a_jx_j(x_j^{q^{i_j}}-x_j)-\lambda$ with $\lambda \in \fqn$, $a_j \in \fq^*$ and $0<i_j<n$ and satisfying the same conditions  as in Theorem \ref{th1}. 

The hypersurface $\mathcal H_r$ attains the upper Weil bound if and only if one of the following holds
\begin{itemize}
\item $\Tr(\lambda) =0$, $D_1=0$, $nr$ is even, $i_j = d_j$ for all $j \in Y$,  $(nr-2D_2)s \equiv 0 \pmod 4$ and $(-1)^{D_2}\chi((-1)^{D_2}AL_1) =1;$
\item  $\Tr(\lambda) =0$, $D_1=0$, $nr$ is even, $i_j = d_j$ for all $j \in Y$ , $(nr-2D_2)s \equiv 2 \pmod 4$ and $(-1)^{D_2}\chi((-1)^{D_2}AL_1) =-1.$
\end{itemize}

The hypersurface $\mathcal H_r$ attains the lower bound if and only if one of the following holds
\begin{itemize}
\item $\Tr(\lambda) =0$, $D_1=0$, $nr$ is even, $i_j = d_j$ for all $j \in Y$,  $(nr-2D_2)s \equiv 2 \pmod 4$ and $(-1)^{D_2}\chi((-1)^{D_2}AL_1) =1;$
\item  $\Tr(\lambda) =0$, $D_1=0$, $nr$ is even, $i_j = d_j$ for all $j \in Y$ , $(nr-2D_2)s \equiv 0 \pmod 4$ and $(-1)^{D_2}\chi((-1)^{D_2}AL_1) =-1.$
\end{itemize}

\end{theorem} 
\begin{example}
Let $q=5^2$, $n=30$. We consider the Artin-Schreier hypersurface given by
\[\mathcal H: y^q-y = x_1(x_1^{q^2}-x_1)+x_2(x_2^{q^3}-x_2)+x_3(x_3^{q^6}-x_3).\]
Following the notation of Theorem \ref{HW}, we have that $i_1=d_1=2, i_2=d_2=3, i_3=d_3=6$, $l_1=15, l_2=10, l_3=5$, $D_2 = 2+3+6 = 11$, $A =1$  and $L_1=1$. Moreover,
\[(nr-2D_2)s\equiv (90-22)\cdot 2 \equiv  0 \pmod 4,\]
\[(-1)^{D_2}\chi((-1)^{D_2}AL_1) =  (-1)^{11} \chi((-1)^{11}) = -1.\]

It follows from Theorem \ref{HW} that $\mathcal H$ is $\mathbb F_{q^{30}}$-minimal. 
\end{example}

\end{document}